\documentclass{amsart}

\usepackage{color,graphicx,amssymb,latexsym,amsfonts,txfonts,amsmath,amsthm}
\usepackage{pdfsync}
\usepackage{amscd}
\usepackage[all,cmtip]{xy}

\usepackage{hyperref}
\hypersetup{
    colorlinks=true,       
    linkcolor=blue,          
    citecolor=blue,        
    filecolor=blue,      
    urlcolor=blue           
}

\input epsf

\usepackage{graphics}
\usepackage{tikz}
\usetikzlibrary{arrows,decorations.markings,shapes.geometric}
\tikzstyle{point}=[circle, inner sep = 2pt, outer sep = 1pt, minimum size = 5pt, fill=black, draw=black]

\newtheorem{theo}{Theorem}

\newtheorem{lemma}{Lemma}

\newtheorem{rema}{Remark}

\begin{document}

\title{Realising countable groups as automorphisms of origamis on the Loch Ness monster}
\author[R.A. Hidalgo]{Rub\'en A. Hidalgo}
\author[I. Morales]{Israel Morales}

\subjclass[2010]{}

\address{Departamento de Matem\'atica y Estad\'{\i}stica, Universidad de La Frontera, Temuco, Chile}
\email{ruben.hidalgo@ufrontera.cl}

\address{Instituto de Matem\'aticas UNAM Unidad Oaxaca, Oaxaca, M\'exico}
\email{imorales@im.unam.mx}

\thanks{The first author is partially supported by Projects Fondecyt 1190001 and 1220261. The second author was partially supported by Proyecto CONACyT Ciencia de Frontera 217392 cerrando brechas y extendiendo puentes en Geometr\'ia y Topolog\'ia.}

\begin{abstract}
It is known that every finite group can be represented as the full group of automorphisms of a suitable compact origami. In this paper, we provide a short argument to note  that the same holds for any countable group by considering origamis on the Loch Ness monster. 
\end{abstract}

\maketitle

\section{Introduction}
In this paper, all surfaces are assumed to be orientable, connected, Hausdorff, second countable and without boundary. If $X$ is a surface, then we denote by ${\rm Hom}^{+}(X)$ the group of its orientation-preserving self-homeomorphisms. The Loch Ness monster is the unique, up to homeomorphisms, infinite genus surface with exactly one end.

Following \cite{Lochak}, an origami ${\mathcal O}$ on a surface $X$, consists of a countable collection of copies of the Euclidean unit square which are glued by the following rules:
\begin{enumerate}
\item each left edge of a square is identified by a translation with a right one;
\item each lower edge of a square is identified by a translation with an upper one;
\item after gluing all the unit squares, one obtains, up to homeomorphisms, $X$.
\end{enumerate}

For details on compact origamis, the reader may consult \cite{HS,S}. Some works in the setting of non-compact origamis are \cite{Chamanara04,BV13,HHW13,Hooper15}. 

\begin{rema}
Note that, in our definition above, the orbit of every vertex, under such identifications, is necessarily finite. If the orbit of a vertice were infinite, then we should have deleted such an orbit and it will provide a point in one of the ends of the surface. 
\end{rema}

The origami ${\mathcal O}$ induces an embedded connected graph ${\mathcal G}$ on $X$, whose edges are the sides of the squares, the vertices are the vertices of the squares and the faces (i.e., the complement components) are squares (with certain extra properties, such as that about each vertex one sees a multiple of four squares). We say that $\phi \in {\rm Hom}^{+}(X)$ induces an automorphism of the origami ${\mathcal O}$, if it preserves the corresponding square tessellation. In this case, $\phi$ induces an automorphism $\rho$ of the graph ${\mathcal G}$; such $\rho$ is called an automorphism of ${\mathcal O}$. We denote by ${\rm Aut}({\mathcal O})$ the group of these automorphisms. If $X$ is compact, then ${\rm Aut}({\mathcal O})$ is necessarily finite. If $X$ is non-compact, then ${\rm Aut}({\mathcal O})$ might be infinite (but countable).

In \cite{H} it was observed that  every finite group is isomorphic to the group of automorphisms of some compact origami. In the following, we 
extends the above for any countable group.

\begin{theo}\label{main}
Let $G$ be a countable (finite or infinite) group. Then, there is an origami ${\mathcal O}$ such that
$Aut(\mathcal{O}) \cong G$. Let $X$ be the underlying surface of the origami. If $G$ is finite, then $X$ can be assumed to be compact and, if $G$ is infinite, then $X$ can be assumed to be homeomorphic to a Loch Ness monster.
\end{theo}

\begin{rema}
In \cite{H} it was proved that every topological action of a finite group, with positive genus quotient, can be realized by a group of automorphisms of a compact origami. We wonder if this more general fact is still valid at the level of any countable group action of positive genus quotient.
\end{rema}

\section{Preliminaries}\label{Sec:prelim}

\subsection{Locally finite branched coverings}
A surjective continuous map $Q:X \to Y$, where $X$ and $Y$ are surfaces, is called a locally finite branched cover if there exists a discrete subset $B_{Q} \subset Y$ (which it might be empty), called the locus of branched values of $Q$, such that 
\begin{enumerate}
\item $Q:X \setminus Q^{-1}(B_{Q}) \to Y \setminus B_{Q}$ is a covering map (orientation-preserving); and
\item each point $q \in B_{Q}$ has an open connected neighbourhood $U$ such that $Q^{-1}(U)$ consists of a collection $\{V_{j}\}_{j\in I}$ of pairwise disjoint connected open sets such that each of the restrictions $Q|_{V_j}:V_{j} \to U$  is a finite degree branched cover (i.e., is topologically equivalent to a branched cover of the form $z \in {\mathbb D} \mapsto z^{d_{j}} \in {\mathbb D}$, where ${\mathbb D}$ denotes the unit disc).
\end{enumerate}

If $X$ is a compact surface, then a branched covering (i.e., just satisfying (1) and (2)) is automatically locally finite. But, in the case that $X$ is non-compact, a branched covering might not be locally finite.

As usual, the deck group of a locally finite branched cover $Q:X \to Y$ is given by ${\rm deck}(Q):=\{ \phi \in {\rm Hom}^{+}(X): Q \circ \phi=Q\}$.

Let $G$ be a group, $X$ be a surface and $\theta:G \to {\rm Hom}^{+}(X)$ be an injective homomorphism. We say that the action of $\theta(G)$ is tame on $X$ if there is a Galois locally finite branched cover $Q:X \to Y$ with ${\rm deck}(Q)= \theta(G)$ (in particular, $G$ is countable).

If $Q:X \to Y$ is a locally finite branched cover, with $A={\rm deck}(Q)$, then there are a Galois locally finite branched cover $\pi_{A}:X \to R$, with ${\rm deck}(\pi_{A})=A$, and a locally finite branched cover $L:R \to Y$ such that $Q=L \circ \pi_{A}$. Moreover, if $H<A$, then there are a Galois locally finite branched cover $\pi_{H}:X \to Z$, with ${\rm deck}(\pi_{H})=H$, and a a locally finite branched cover $P:Z \to R$ such that $\pi_{A}=P \circ \pi_{H}$.

\subsection{Origami pairs}
Origamis can be equivalently be defined as certain types of locally finite branched covers of a two-dimensional torus, with at most one branch value.

Let $T \cong S^{1} \times S^{1}$ be the $2$-dimensional torus. Chose a point $p_{0} \in T$ and two simple loops $\alpha, \beta \subset T$  based at $p_{0}$ and inducing a canonical set of generators of $\pi_{1}(T,p_{0})$ (geometrically, by cutting $T$ along these two loops, we obtain an square where its four vertices correspond to $p_{0}$).

An origami pair is a pair $(X,Q)$, where $X$ is a surface and $Q:X \to T$ is a locally finite branched covering whose branched value set is contained inside $\{p_{0}\}$. If $X$ has negative Euler characteristic, then $Q$ must have a branch value.

If $(X,Q)$ is an origami pair, then the graph $Q^{-1}(\{\alpha, \beta\})={\mathcal G}$ defines an origami ${\mathcal O}$ on $X$ and, conversely, every origami can be so constructed. 
Moreover, 
${\rm Aut}({\mathcal O})={\rm deck}(Q)$.

\subsection{An auxiliary origami on the Loch Ness monster}
Each origami can be described by two permutations $\sigma, \tau \in {\rm Sym}(\Omega)$, where $\Omega$ denotes the collection of squares used to construct the given origami. The permutation $\sigma$ (respectively, $\tau$) indicates the horizontal (respectively, vertical) gluing. The connectivity property is equivalent for the group generated by $\sigma$ and $\tau$ to be a transitive subgroup of ${\rm Sym}(\Omega)$. A necessary and sufficient condition, for $\sigma$ and $\tau$, to define an origami (in our definition) is that the commutator $[\sigma,\tau]$ is a product of disjoint cycles of finite lengths.

\begin{lemma}\label{lema1}
There is an origami pair $(X,Q)$, where $X$ is the Loch Ness monster, with branch value $p_{0}$, such that in $Q^{-1}(p_{0})$ there is exactly one point with local degree one.
\end{lemma}
\begin{proof}
Let us consider a collection $\{R_{i}: i \in {\mathbb N}\}$ of copies of unit squares in the plane ${\mathbb R}^{2}$ and consider the gluing induced by the permutations of ${\mathbb Z}$ defined by
$$\sigma=(1)(2)(3,4)(5)(6,7)(8)(9,10)(11)(12,13)\cdots$$
$$\tau=(1,2,3)(4,5,6)(7,8,9)(10,11,12)\cdots$$

Now, perform the horizontal gluing of the squares $R_{i}$ following $\sigma$ and the vertical gluings following $\tau$ (see Figure \ref{figura1}). The two vertex sharing the squares $1$ and $2$ provide the vertice of local degree one. There is another vertice of local degree two and all others have local degree three. 
\end{proof}

\begin{figure}[ht]\label{figura1}
\setlength{\unitlength}{1540sp}
\begingroup\makeatletter\ifx\SetFigFont\undefined%
\gdef\SetFigFont#1#2#3#4#5{%
  \reset@font\fontsize{#1}{#2pt}%
  \fontfamily{#3}\fontseries{#4}\fontshape{#5}%
  \selectfont}%
\fi\endgroup%
\centering
\begin{picture}(3024,6024)(5464,-6373)
\put(6901,-3136){\makebox(0,0)[lb]{\smash{{\SetFigFont{10}{20.4}{\rmdefault}{\mddefault}{\updefault}{\color[rgb]{0,0,0}8}%
}}}}
\thinlines
{\color[rgb]{0,0,0}\put(6076,-5161){\framebox(600,600){}}
}%
{\color[rgb]{0,0,0}\put(5476,-5161){\framebox(600,600){}}
}%
{\color[rgb]{0,0,0}\put(5476,-5761){\framebox(600,600){}}
}%
{\color[rgb]{0,0,0}\put(5476,-6361){\framebox(600,600){}}
}%
{\color[rgb]{0,0,0}\put(6076,-3961){\framebox(600,600){}}
}%
{\color[rgb]{0,0,0}\put(6676,-3961){\framebox(600,600){}}
}%
{\color[rgb]{0,0,0}\put(6676,-3361){\framebox(600,600){}}
}%
{\color[rgb]{0,0,0}\put(6676,-2761){\framebox(600,600){}}
}%
{\color[rgb]{0,0,0}\put(7276,-2761){\framebox(600,600){}}
}%
{\color[rgb]{0,0,0}\put(7276,-2161){\framebox(600,600){}}
}%
{\color[rgb]{0,0,0}\put(7276,-1561){\framebox(600,600){}}
}%
{\color[rgb]{0,0,0}\put(7876,-1561){\framebox(600,600){}}
}%
{\color[rgb]{0,0,0}\put(7876,-961){\framebox(600,600){}}
}%
\put(5701,-5536){\makebox(0,0)[lb]{\smash{{\SetFigFont{10}{20.4}{\rmdefault}{\mddefault}{\updefault}{\color[rgb]{0,0,0}2}%
}}}}
\put(5701,-6136){\makebox(0,0)[lb]{\smash{{\SetFigFont{10}{20.4}{\rmdefault}{\mddefault}{\updefault}{\color[rgb]{0,0,0}1}%
}}}}
\put(5701,-5011){\makebox(0,0)[lb]{\smash{{\SetFigFont{10}{20.4}{\rmdefault}{\mddefault}{\updefault}{\color[rgb]{0,0,0}3}%
}}}}
\put(6301,-4936){\makebox(0,0)[lb]{\smash{{\SetFigFont{10}{20.4}{\rmdefault}{\mddefault}{\updefault}{\color[rgb]{0,0,0}4}%
}}}}
\put(6301,-4336){\makebox(0,0)[lb]{\smash{{\SetFigFont{10}{20.4}{\rmdefault}{\mddefault}{\updefault}{\color[rgb]{0,0,0}5}%
}}}}
\put(6301,-3736){\makebox(0,0)[lb]{\smash{{\SetFigFont{10}{20.4}{\rmdefault}{\mddefault}{\updefault}{\color[rgb]{0,0,0}6}%
}}}}
\put(6901,-3736){\makebox(0,0)[lb]{\smash{{\SetFigFont{10}{20.4}{\rmdefault}{\mddefault}{\updefault}{\color[rgb]{0,0,0}7}%
}}}}
\put(6901,-2536){\makebox(0,0)[lb]{\smash{{\SetFigFont{10}{20.4}{\rmdefault}{\mddefault}{\updefault}{\color[rgb]{0,0,0}9}%
}}}}
\put(7426,-2536){\makebox(0,0)[lb]{\smash{{\SetFigFont{10}{20.4}{\rmdefault}{\mddefault}{\updefault}{\color[rgb]{0,0,0}10}%
}}}}
\put(7426,-1936){\makebox(0,0)[lb]{\smash{{\SetFigFont{10}{20.4}{\rmdefault}{\mddefault}{\updefault}{\color[rgb]{0,0,0}11}%
}}}}
\put(7426,-1336){\makebox(0,0)[lb]{\smash{{\SetFigFont{10}{20.4}{\rmdefault}{\mddefault}{\updefault}{\color[rgb]{0,0,0}12}%
}}}}
\put(8001,-1336){\makebox(0,0)[lb]{\smash{{\SetFigFont{10}{20.4}{\rmdefault}{\mddefault}{\updefault}{\color[rgb]{0,0,0}13}%
}}}}
\put(8001,-811){\makebox(0,0)[lb]{\smash{{\SetFigFont{10}{20.4}{\rmdefault}{\mddefault}{\updefault}{\color[rgb]{0,0,0}14}%
}}}}
{\color[rgb]{0,0,0}\put(6076,-4561){\framebox(600,600){}}
}%
\end{picture}%

\caption{Lower part of the origami as defined in Lemma \ref{lema1}}
\end{figure}
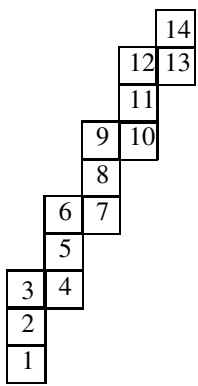

\section{Proof of Theorem \ref{main}}
\begin{proof}
We only need to take care of the case when $G$ is infinite group, as the finite case is done in \cite{H}.
As a consequence of \cite[Theorem 7.2]{APV} (by taking $S=X$ the Loch Ness monster), there exists a Galois covering $\pi:X \to X$ with deck group 
$G \cong \theta(G) < {\rm Hom}^{+}(X)$. In particular, the action of $\theta(G)$ is a free tame action (this can also be seen by providing a Riemann surface structure on $X$ and lifting it under $\pi$ to obtain a Riemann surface structure on the top $X$. In this way, $\theta(G)$ is a discete group of conforal automorphisms of the last Riemann surface structure).

Following Lemma \ref{lema1}, we may consider an origami pair $(X,Q_{0})$, with branch value $p_{0} \in T$, and such that in $Q_{0}^{-1}(p_{0})$ there is exactly one vertex $x_{0}$ with local degree one.

Let us consider the origami ${\mathcal O}$ induced by the origami pair $(X,Q_{0} \circ \pi)$. We observe that 
$\theta(G) \leq A={\rm Aut}({\mathcal O})$.
We claim that $A=\theta(G)$. Let us assume, by the contrary, that $\theta(G) \neq A$. In this case, let us consider a Galois (possible branched) cover $\pi_{A}:X \to R$, with deck group $A$ (so $\pi_{A}$ is a locally finite branched cover). We may observe the following.
\begin{enumerate}
\item[(i)]  As $\theta(G) \leq A$ and as $\pi$ and $\pi_{A}$ are locally finite branched covers, there is a locally finite branched cover  $P:X \to R$, degree at least two (it could be of infinite degree) such that $P \circ \pi=\pi_{A}$, and
\item[(iii)] As $A$ is contained in the deck group of $Q_{0} \circ \pi$, there is a locally finite branched cover  $L:R \to T$ such that $L \circ P=Q_{0}$.
\end{enumerate}

Set $r_{0}=P(x_{0})$ and let $d \geq 1$ be the branched order of $\pi_{A}$ at $r_{0}$. Observe that each point in the fiber $\pi_{A}^{-1}(r_{0})$ has as $A$-stabilizer a cyclic group of order $d$.

As $\pi$ has no branched values, it follows that the local degree of $P$ at each point in the fiber $P^{-1}(r_{0})$ is $d$. But this means that, at each of these points, the local degree of $Q_{0}=L \circ P$ must be the same. This is a contradiction to condition that $Q_{0}$ has exactly one point of local degree one in $Q_{0}^{-1}(p_{0})$. 
\end{proof}


\end{document}